\documentclass[12pt]{amsart}
\usepackage{amsmath}
\usepackage{amssymb}
\usepackage{amsthm}
\usepackage{amscd}
\usepackage{enumerate}
\usepackage{enumitem}
\usepackage{verbatim}
\usepackage[all]{xy}
\usepackage{mathrsfs}
\usepackage{mathabx}
\usepackage{url}
\usepackage{here}

\theoremstyle{plain}
\newtheorem{thm}{Theorem}[section]
\newtheorem{prop}[thm]{Proposition}
\newtheorem{lem}[thm]{Lemma}
\newtheorem{cor}[thm]{Corollary}

\theoremstyle{definition}
\newtheorem{dfn}[thm]{Definition}
\newtheorem{rmk}[thm]{Remark}

\newcommand{\Cusps}{\mathrm{Cusps}}

\newcommand{\End}{\mathrm{End}}

\newcommand{\har}{\mathrm{har}}

\newcommand{\Hom}{\mathrm{Hom}}

\newcommand{\Ker}{\mathrm{Ker}}

\newcommand{\ord}{\mathrm{ord}}

\newcommand{\Res}{\mathrm{Res}}

\newcommand{\Spec}{\mathrm{Spec}}

\newcommand{\St}{\mathrm{St}}
\newcommand{\Stab}{\mathrm{Stab}}
\newcommand{\src}{\mathrm{src}}

\newcommand{\cO}{\mathcal{O}}

\newcommand{\cT}{\mathcal{T}}

\newcommand{\Ga}{\mathbb{G}_\mathrm{a}}

\newcommand{\bC}{\mathbb{C}}

\newcommand{\bF}{\mathbb{F}}

\newcommand{\bP}{\mathbb{P}}

\newcommand{\bZ}{\mathbb{Z}}

\newcommand{\fra}{\mathfrak{a}}

\newcommand{\frem}{\mathfrak{m}}
\newcommand{\frn}{\mathfrak{n}}

\makeatletter
    
    \@addtoreset{equation}{section}
\makeatother

\makeatletter
\renewcommand{\p@enumii}{}
\makeatother

\begin{document}

\title[Triviality of the Hecke action on Drinfeld cuspforms]{Triviality of the Hecke action on ordinary Drinfeld cuspforms of level $\Gamma_1(t^n)$}
\author{Shin Hattori}
\address[Shin Hattori]{Department of Natural Sciences, Tokyo City University}
%\email{hattoris@tcu.ac.jp}
%\affil{Department of Natural Sciences, Tokyo City University}

\date{\today}

%\classification{11F52}
%\keywords{Drinfeld modular form, Hecke eigenvalue, ordinary}
%\thanks{Supported by JSPS KAKENHI Grant Numbers JP17K05177 and JP20K03545.}

\begin{abstract}
%Let $p$ be a rational prime, $q>1$ a $p$-power integer and $A=\bF_q[t]$. 
Let $k\geq 2$ and $n\geq 1$ be any integers. In this paper, we prove that all Hecke operators
act trivially on the space of ordinary Drinfeld cuspforms of level $\Gamma_1(t^n)$ and weight $k$.

\end{abstract}

\maketitle
%\tableofcontents

%---------------------------------------------------------------------

%---------------------------------------------------------------------

\section{Introduction}\label{SecIntro}

Let $p$ be a rational prime, $q>1$ a $p$-power integer, $A=\bF_q[t]$, $K=\bF_q(t)$ and $K_\infty=\bF_q((1/t))$. 
Let $\bC_\infty$ be the $(1/t)$-adic completion of an algebraic closure of $K_\infty$ and put $\Omega=\bC_\infty\setminus K_\infty$, which has a natural 
structure as a rigid analytic variety over $K_\infty$. For any non-zero element $\frn\in A$, we put
\[
\Gamma_1(\frn)=\left\{\gamma\in \mathit{SL}_2(A)\ \middle |\ \gamma\equiv\begin{pmatrix}
1 & *\\
0& 1
\end{pmatrix}
\bmod \frn \right\}.
\]

For any arithmetic subgroup $\Gamma$ of $\mathit{SL}_2(A)$ and integer $k\geq 2$, a rigid analytic function $f:\Omega\to \bC_\infty$ is called a Drinfeld 
modular form of level $
\Gamma$ and weight $k$ if it satisfies
\[
f\left(\frac{az+b}{cz+d}\right)=(cz+d)^kf(z)\quad \text{for any }\begin{pmatrix}
a & b\\
c& d
\end{pmatrix}\in \Gamma
\]
and a certain regularity condition at cusps. A Drinfeld modular form is called a cuspform if it vanishes at all cusps, and a double cuspform 
if it vanishes twice at all cusps.
They form $\bC_\infty$-vector spaces $S_k(\Gamma)$ and $S^{(2)}_k(\Gamma)$, respectively. These spaces admit a natural action of Hecke operators.

Let $\wp\in A$ be an irreducible polynomial, $K_\wp$ the $\wp$-adic completion of $K$ and $\bC_\wp$ the $\wp$-adic completion of an algebraic closure of $K_\wp$.
For the algebraic closure $\bar{K}$ of $K$ in $\bC_\infty$, we fix an embedding $\iota_\wp: \bar{K}\to \bC_\wp$.

Suppose that $\wp$ divides $\frn$. The Hecke 
operator at $\wp$ acting on $S_k(\Gamma_1(\frn))$ is denoted by $U_\wp$. Note that any eigenvalue of 
$U_\wp$ is an element 
of $
\bar{K}$.
We say $f\in S_k(\Gamma_1(\frn))$ is ordinary (with respect to $\iota_\wp$) if $f$ is in the generalized eigenspace belonging to an eigenvalue $\lambda\in 
\bar{K}$ satisfying $\iota_\wp(\lambda)\in \cO_{\bC_\wp}^\times$. We denote the subspace of ordinary Drinfeld cuspforms by $S_k^\ord(\Gamma_1(\frn))$. 
It is an analogue of the notion of ordinariness for elliptic modular forms studied in \cite{Hida_E}.

Let us focus on the case $\frn=t^n$ and $\wp=t$ with some integer $n\geq 1$. In this case, the structure of 
$S_k^\ord(\Gamma_1(t^n))$ seems quite simple.
For $n=1$, it is known that all Hecke operators act trivially on the one-dimensional $\bC_\infty$-vector space $S_k^\ord(\Gamma_1(t))$ \cite[Proposition 
4.3]{Ha_DMBF}. In this paper, we prove that this holds in general, as follows.

\begin{thm}[Theorem \ref{ThmTriv}]\label{ThmIntro}
	Let $k\geq 2$ and $n\geq 1$ be any integers. Then we have 
	\[
	\dim_{\bC_\infty}S_k^\ord(\Gamma_1(t^n))=q^{n-1}
	\]
	 and all Hecke operators act 
	trivially on 
	$S_k^\ord(\Gamma_1(t^n))$.
	\end{thm}
This suggests that Hida theory for Drinfeld cuspforms should be trivial for the level $\Gamma_1(t^n)$.

For Drinfeld modular forms, it is well-known that the weak multiplicity one, which states that any Hecke eigenform is determined up to a scalar multiple 
by the Hecke eigenvalues, is false. Gekeler \cite[\S7]{Gek} raised a question if the property holds when we fix the weight. Theorem \ref{ThmIntro} gives a 
negative answer to 
it (see also \cite[Examples 15.4 and 15.7]{Boeckle} for a variant ignoring Hecke eigenvalues at places dividing the level).

For the proof of Theorem \ref{ThmIntro}, we study a 
subspace $S'_k$ of $S_k=S_k(\Gamma_1(t^n))$ containing $S_k^{(2)}=S_k^{(2)}(\Gamma_1(t^n))
$. It consists of cuspforms which vanish twice at unramified cusps (\S\ref{SubsecDDC}). We show that all Hecke 
operators act trivially on $S_k/S'_k$ and $U_t$ is nilpotent on $S'_k/S_k^{(2)}$ (Lemma \ref{LemQuotTriv} and 
Proposition \ref{PropMidNilp}). Then, using the constancy of the dimension of $S_k^\ord(\Gamma_1(t^n))$ with respect to $k$ \cite[Proposition 3.4]{Ha_DMBF}, 
we reduce Theorem \ref{ThmIntro} to showing that the dimension of 
$S_2^\ord(\Gamma_1(t^n))$ is no more than $q^{n-1}$
(Theorem \ref{ThmNilpV}).

Consider the multiplicative group $\Theta_n=1+tA/t^n A$, which acts on $S_k(\Gamma_1(t^n))$ via the diamond operator. 
To obtain the upper bound of the dimension, the key point is the 
freeness 
of $S_2(\Gamma_1(t^n))$ as a module over the group ring
$\bC_\infty[\Theta_n]$ (Proposition \ref{PropFree}): From the fact that $S_2^\ord(\Gamma_1(t))$ is one-dimensional \cite[Lemma 2.4]{Ha_GVt}
and another constancy result of the dimension of the ordinary subspace \cite[Proposition 3.5]{Ha_DMBF},
we see that the $\Theta_n$-fixed part of $S_2^\ord(\Gamma_1(t^n))$ is also one-dimensional. Thus the freeness implies 
that it injects into a single component $\bC_\infty[\Theta_n]$ of the free $\bC_\infty[\Theta_n]$-module $S_2(\Gamma_1(t^n))$, 
which gives the desired bound.

The paper is organized as follows. In \S\ref{SecDCF}, we will recall the definition of Hecke operators and study their effect on Fourier expansions of Drinfeld 
cuspforms at cusps. In 
\S\ref{SecU}, we will define the 
subspace $S'_k$ and study its properties analytically.
In \S\ref{SecTriv}, using the description of Drinfeld cuspforms via harmonic cocycles on the Bruhat-Tits tree 
\cite{Tei,Boeckle}, we will give an explicit basis 
of the $\bC_\infty$-vector space $S_2(\Gamma_1(t^n))$ and a description of the diamond operator in terms of the basis. 
These enable us to show the freeness and Theorem \ref{ThmIntro}.

\subsection*{Acknowledgements} The author would like to thank Ernst-Ulrich Gekeler and Federico Pellarin for helpful conversations on this topic, 
and Gebhard B\"{o}ckle for pointing out an error in a previous manuscript.
A part of this work was carried out during the author's visit to Universit\'{e} Jean Monnet. He wishes to thank its hospitality. This work was supported 
by JSPS KAKENHI Grant Numbers JP17K05177 and JP20K03545.

%---------------------------------------------------------------------
%---------------------------------------------------------------------

\section{Drinfeld cuspforms of level $\Gamma_1(\frn)$}\label{SecDCF}

%---------------------------------------------------------------------
%---------------------------------------------------------------------

Let $k\geq 2$ be any integer and $\frn$ any element of $A\setminus \bF_q$. In this section, we study Hecke operators acting 
on $S_k(\Gamma_1(\frn))$.
For any group $\Gamma$ acting on a set $X$, we denote the stabilizer of $x\in X$ in $\Gamma$ by $
\Stab(\Gamma,x)$.

\subsection{Cusps and uniformizers}\label{SubsecCU}

%---------------------------------------------------------------------
%---------------------------------------------------------------------

Consider the action of $\mathit{SL}_2(A)$ on $\bP^1(\bC_\infty)$ defined by
\[
\begin{pmatrix}
	a&b\\c&d
\end{pmatrix}
\begin{pmatrix}
	x\\ y
\end{pmatrix}=\begin{pmatrix}
	ax+by\\ cx+dy
\end{pmatrix}.
\]
We refer to any element of $\bP^1(K)$ as a cusp.
For any arithmetic subgroup $\Gamma$ of $\mathit{SL}_2(A)$, put
\[
\Cusps(\Gamma)=\Gamma\backslash \bP^1(K).
\]
We abusively identify an element of $\Cusps(\Gamma)$ with a cusp representing it.

Next we recall the definition of the uniformizer at each cusp \cite[(2.7)]{GR}, following the normalization of \cite[(4.1)]{Gek}. 
Let $C$ be the Carlitz module. It is the Drinfeld module of rank one over $A$ defined by
the homomorphism of $\bF_q$-algebras
\[
A=\bF_q[t]\to \End(\Ga),\quad t\mapsto (Z\mapsto tZ+Z^q),
\]
where we put $\Ga=\Spec(A[Z])$. For any $a\in A$, we denote by $\Phi^C_a(Z)$ the element of $A[Z]$ such that
the image of $a$ by the map above is defined by $(Z\mapsto \Phi^C_a(Z))$.

For any subgroup $\mathfrak{b}$ of $A$ containing a non-zero ideal of $A$, we define
\[
e_{\mathfrak{b}}(z)=z\prod_{0\neq b\in \mathfrak{b}}\left(1-\frac{z}{b}\right),
\]
which is an entire function on $\Omega$. Let $\bar{\pi}\in \bC_\infty$ be a Carlitz period, so that
\begin{equation}\label{EqnCar}
	\Phi^C_t(\bar{\pi}e_A(z))=\bar{\pi}e_A(tz).
\end{equation}
For any integer $l\geq 0$, we put
\[
u_\mathfrak{b}(z)=\frac{1}{\bar{\pi}e_{\mathfrak{b}}(z)},\quad u(z)=u_A(z),\quad u_l(z)=u_{(t^l)}(z)=\frac{1}{t^l}u\left(\frac{z}{t^l}\right).
\]

Since $\frn\in A\setminus \bF_q$, the group $\Gamma_1(\frn)$ is $p'$-torsion free.
For any cusp $s\in \bP^1(K)$, choose $
\nu_s\in \mathit{SL}_2(A)$ satisfying $\nu_s(\infty)=s$ and put
\[
\mathfrak{b}_s=\left\{b\in A\ \middle|\ \begin{pmatrix}
	1&b\\0&1
\end{pmatrix}\in \Stab(\nu^{-1}_s\Gamma_1(\frn)\nu_s,\infty)\right\}\supseteq (\frn).
\] 
Then we refer to the function 
\[
u_s(z):=u_{\mathfrak{b}_s}(z)
\]
as the uniformizer at $s$ for $\Gamma_1(\frn)$. 
Note that $\mathfrak{b}_s$ depends only on $s$ up to a multiple of an element of $\bF_q^\times$. 
Thus $\mathfrak{b}_s$ and $u_s(z)$ are independent of the choice of $\nu_s$ if $\mathfrak{b}_s$ is an ideal of $A$ for some choice of $\nu_s$.
For example, we have $\mathfrak{b}_\infty=A$ for any choice of $\nu_\infty$ and the uniformizer at $\infty$ is $u(z)$.

For any function $f$ on $\Omega$, integer $k\geq 2$ and $\gamma\in \mathit{GL}_2(K)$,
we define the slash operator by
\[
\left(f|_k\gamma\right)(z)=\det(\gamma)^{k-1}(cz+d)^{-k}f\left(\frac{az+b}{cz+d}\right),\quad \gamma=\begin{pmatrix}
	a&b\\c&d
\end{pmatrix}.
\]
Then, for any $f\in S_k(\Gamma_1(\frn))$, we can write
\[
(f|_k\nu_s)(z)=\sum_{i\geq 1} a_i u_s(z)^i,\quad a_i\in \bC_\infty
\] 
when the $(1/t)$-adic absolute value $|u_s(z)|$ of $u_s(z)$ is sufficiently small. 
We refer to it as the Fourier expansion of $f$ at the cusp $s$
and put
\[
\ord(s,f)=\min\{i\geq 1\mid a_i\neq 0\}.
\]
The latter is independent of the choice of $\nu_s$.

\begin{lem}\label{LemCoeffHecke}
	Let $\frem\in A$ be any monic irreducible polynomial and $i\geq 1$ any integer.
	\begin{enumerate}
		\item\label{LemCoeffHecke-beta1} $\sum_{\deg(\beta)<\deg(\frem)} u\left(\frac{z+\beta}{\frem}\right)=\frem u(z)$.
		\item\label{LemCoeffHecke-beta2} If $i\geq 2$, then 
		$\sum_{\deg(\beta)<\deg(\frem)} u\left(\frac{z+\beta}{\frem}\right)^i\in \frem u(z)^2A[u(z)]$.
		\item\label{LemCoeffHecke-dia} $u(\frem z)\in  u(z)^2 A[[u(z)]].$
	\end{enumerate}
	Here the sum $\sum_{\deg(\beta)<\deg(\frem)}$ runs over the set of $\beta\in A$ satisfying $\deg(\beta)<\deg(\frem)$.
\end{lem}
\begin{proof}
	Put $r=\deg(\frem)$ and $\Phi^C_\frem(Z)=\frem Z+c_1 Z^q+\cdots+c_{r-1}Z^{q^{r-1}}+Z^{q^r}$. 
	Then we have
	\begin{equation}\label{EqnGekelerPre}
	z\prod_{0\neq b\in C[\frem](\bC_\infty)}\left(1-\frac{z}{b}\right)=\frem^{-1}\Phi_\frem^C(z).
	\end{equation}
	Let $\alpha_i$ be the coefficient of $Z^{q^i}$ 
	in $\frem^{-1}\Phi^C_\frem(Z)$.
	By \cite[Lemma 3.2]{Ha_HTT}, we have $\alpha_i\in A$ for any $0\leq i\leq r-1$ and $\alpha_r=\frem^{-1}$.

	Let $G_{i,\frem}(X)$ be the $i$-th Goss polynomial with respect to the $\bF_q$-vector space $C[\frem](\bC_\infty)$. Then \cite[computation above (7.3)]{Gek} 
	gives
	\begin{equation}\label{EqnGekeler}
	\sum_{\deg(\beta)<\deg(\frem)} u\left(\frac{z+\beta}{\frem}\right)^i=G_{i,\frem}(\frem u(z)).
	\end{equation}
	
	For $i=1$, we have $G_{i,\frem}(X)=X$ and (\ref{LemCoeffHecke-beta1}) follows. 
	For $i\geq 2$, \cite[(3.8)]{Gek} and (\ref{EqnGekelerPre}) show that 
	$G_{i,\frem}(X)$ has no linear term and $G_{i,\frem}(\frem X)\in \frem A[X]$, which yields (\ref{LemCoeffHecke-beta2}). 
	Moreover, we have
	\[
	u(\frem z)=\frac{u(z)^{q^r}}{1+c_{r-1}u(z)^{q^r-q^{r-1}}+\cdots+\frem u(z)^{q^r-1}},
	\]
	which implies (\ref{LemCoeffHecke-dia}).
\end{proof}

Put $\zeta_{t^l}=\bar{\pi}e_A\left(\frac{1}{t^l}\right)\in \bC_\infty$, so that $\Phi^C_{t^l}(\zeta_{t^l})=0$ by (\ref{EqnCar}). 

\begin{lem}\label{LemUnifHecke}
	Let $l\geq 1$ be any integer. For any $\beta\in \bF_q$, we have 
	\[
	u_{l-1}\left(\frac{z+\beta}{t}\right)\in t u_{l}(z)A[\zeta_{t^l}][[u_{l}(z)]].
	\]
	Here $A[\zeta_{t^l}]$ is the $A$-subalgebra of $\bC_\infty$ generated by $\zeta_{t^l}$.
\end{lem}
\begin{proof}
	This follows from
	\begin{align*}
		u_{l-1}\left(\frac{z+\beta}{t}\right)&=\frac{t}{t^{l}\bar{\pi}e_A\left(\frac{z+\beta}{t^{l}}\right)}=\frac{t}{t^{l}\bar{\pi}e_A\left(\frac{z}{t^{l}}
		\right)}
		\cdot \frac{1}{1+\frac{t^{l}\bar{\pi}e_A\left(\frac{\beta}{t^{l}}\right)}{t^{l}\bar{\pi}
				e_A\left(\frac{z}
				{t^{l}}\right)}}\\
		&=t u_l(z)\cdot \frac{1}{1+t^{l}\beta\zeta_{t^l}u_l(z)}.
	\end{align*}
\end{proof}

%---------------------------------------------------------------------
%---------------------------------------------------------------------

\subsection{Hecke operators}\label{SubsecHeckeOp}

%---------------------------------------------------------------------
%---------------------------------------------------------------------

Now we recall the definition of Hecke operators (for example, see \cite[\S3.1]{Ha_DMBF}).
Let $\frem\in A$ be any monic irreducible polynomial. 
Then the Hecke operator $T_\frem$ acting on $S_k(\Gamma_1(\frn))$ is defined as
\[
T_\frem f=\sum_{\xi} f|_k \xi,
\]
where $\xi$ runs over any complete set of representatives of the coset space
\begin{equation}\label{EqnCoset}
\Gamma_1(\frn)\backslash \Gamma_1(\frn)\begin{pmatrix}
	1 & 0\\0& \frem
\end{pmatrix}\Gamma_1(\frn).
\end{equation}
When $\frem|\frn$, we write $T_\frem$ also as $U_\frem$. 

Let $\fra\in A$ be any element which is prime to $\frn$. 
Take any matrix $\eta_{\fra,\diamond}\in \mathit{SL}_2(A)$ satisfying
\[
\eta_{\fra,\diamond} \equiv \begin{pmatrix}
	* & *\\0 & \fra
\end{pmatrix}\bmod \frn
\]
and put
\[
\xi_{\fra, \diamond}=
\eta_{\fra, \diamond}\begin{pmatrix}
	\fra & 0\\0 & 1
\end{pmatrix}.
\]
Note that we have
\[
\eta_{\fra,\diamond}\Gamma_1(\frn)\eta_{\fra,\diamond}^{-1}=\Gamma_1(\frn),\quad 
\xi_{\fra,\diamond}\Gamma_1(\fra\frn)\xi_{\fra,\diamond}^{-1}\subseteq \Gamma_1(\frn).
\]
Hence we obtain
\begin{equation}\label{EqnStabilization}
	f\in S_k(\Gamma_1(\frn))\Rightarrow f|_k\eta_{\fra,\diamond}\in S_k(\Gamma_1(\frn)),\ f|_k\xi_{\fra,\diamond}\in 
	S_k(\Gamma_1(\fra\frn)).
\end{equation}

For any $\alpha\in (A/(\frn))^\times$, we choose a lift $\fra\in A$ of $\alpha$ and put
\[
\langle \alpha\rangle_\frn f=f|_k\eta_{\fra,\diamond}
\]
for any $f\in S_k(\Gamma_1(\frn))$, which is independent of the choices of $\fra$ and $\eta_{\fra,\diamond}$. Then $\alpha\mapsto \langle \alpha\rangle_\frn$ 
defines an action of 
the group $(A/(\frn))^\times$ on $S_k(\Gamma_1(\frn))$. 

\begin{lem}\label{LemDiaHecke}
	For any $\alpha\in (A/(\frn))^\times$, the diamond operator $\langle \alpha\rangle_\frn$ commutes with all Hecke operators.
\end{lem}
\begin{proof}
	Let $\frem\in A$ be any monic irreducible polynomial.
	First suppose $\frem\mid \frn$. Write
	\[
	\eta_{\fra, \diamond}=\begin{pmatrix}
		S & S'\\T & T'
	\end{pmatrix}
	\]
	with some $S,S',T,T'\in A$ satisfying $T\equiv 0,\ T'\equiv \fra\bmod \frn$ and $ST'-S'T=1$.
	Since $S$ is prime to $\frn$, there exists $\beta\in A$ satisfying $\beta S\equiv S'\bmod \frn$. Then we have
	\[
	\eta_{\fra,\diamond}^{-1}\begin{pmatrix}
		1 & 0\\0& \frem
	\end{pmatrix}\eta_{\fra,\diamond}\in \Gamma_1(\frn)\begin{pmatrix}
	1 & \beta\\0& \frem
\end{pmatrix}
=\Gamma_1(\frn)\begin{pmatrix}
	1 & 0\\0&\frem
\end{pmatrix}\begin{pmatrix}
	1 & \beta\\0&1
\end{pmatrix},
	\]
	which yields
	\begin{equation}\label{EqnDiaCoset}
		\Gamma_1(\frn)\eta_{\fra,\diamond}^{-1}\begin{pmatrix}
			1 & 0\\0& \frem
		\end{pmatrix}\eta_{\fra,\diamond}\Gamma_1(\frn)=
		\Gamma_1(\frn)\begin{pmatrix}
			1 & 0\\0& \frem
		\end{pmatrix}\Gamma_1(\frn).
	\end{equation}
The lemma in this case follows from this equality.
	
	Next suppose $\frem\nmid \frn$.
	Note that the natural map
	\[
	\mathit{SL}_2(A)\to \mathit{SL}_2(A/(\frn))\times \mathit{SL}_2(A/(\frem))
	\]
	is surjective. Since $\langle \alpha\rangle_\frn$ is independent of the choices of $\fra$ and $\eta_{\fra,\diamond}$, we may assume that $\eta_{\fra,
	\diamond}$ satisfies
	\[
	\eta_{\fra,\diamond}\equiv \begin{pmatrix}
		1 & 0\\0& 1
	\end{pmatrix}\bmod \frem.
	\]
	Then we have
	\[
	\eta_{\fra,\diamond}^{-1}\begin{pmatrix}
		1 & 0\\0& \frem
	\end{pmatrix}\eta_{\fra,\diamond}\in \Gamma_1(\frn)\begin{pmatrix}
		1 & 0\\0& \frem
	\end{pmatrix}
	\]
	and (\ref{EqnDiaCoset}) holds also in this case, which yields the lemma.
	\end{proof}

Let us give an explicit description of the Hecke operator $T_\frem$.
For any $\beta\in A$ satisfying $\deg(\beta)<\deg(\frem)$, put
\[
\xi_{\frem,\beta}=\begin{pmatrix}
	1 & \beta\\0 & \frem
\end{pmatrix}.
\]
When $\frem=t$, we also write $\xi_\beta$ for $\xi_{t,\beta}$.
Then the operator $U_\frem$ for $\frem\mid \frn$ is given by
\[
(U_\frem f)(z)=\sum_{\deg(\beta)<\deg(\frem)} (f|_k\xi_{\frem,\beta})(z)=\frac{1}{\frem}\sum_{\deg(\beta)<\deg(\frem)} f\left(\frac{z+\beta}{\frem}\right).
\]

When $\frem\nmid \frn$,  the set
\[
\{\xi_{\frem, \beta}\mid \deg(\beta)<\deg(\frem)\}\cup \{\xi_{\frem, \diamond}\}
\]
forms a complete set of representatives of the coset space (\ref{EqnCoset}) and
thus
\[
T_\frem f=\sum_{\deg(\beta)<\deg(\frem)} f|_k \xi_{\frem, \beta}+f|_k\xi_{\frem, \diamond}.
\]

\section{$U_t$-operator of level $\Gamma_1(t^n)$}\label{SecU}

Let $k\geq 2$ and $n\geq 1$ be any integers. In the rest of the paper, we assume $\frn=t^n$.

In this section, we study the operator $U_t$ acting on $S_k(\Gamma_1(t^n))$, and prove a criterion, in terms of 
$U_t$, for all Hecke 
operators to act trivially on $S_k^\ord(\Gamma_1(t^n))$ (Theorem \ref{ThmNilpV}).
We denote $S_k(\Gamma_1(t^n))$ and $S_k^{(2)}(\Gamma_1(t^n))$ also by $S_k$ and $S_k^{(2)}$, respectively.

Put $A_n=A/(t^n)$. Let $v_t$ be the $t$-adic valuation on $K$ normalized as $v_t(t)=1$. 
For any $c\in A_{n-1}$, take any lift $\tilde{c}\in A$ of $c$ and put
\[
\bar{v}_t(c)=\min\{v_t(\tilde{c}),n-1\},
\]
which is independent of the choice of $\tilde{c}$.

%---------------------------------------------------------------------
%---------------------------------------------------------------------

\subsection{Cusps of $\Gamma_1(t^n)$}\label{SubsecCusps}

%---------------------------------------------------------------------
%---------------------------------------------------------------------

For any $c,d\in A_{n-1}$, put
\[
\bar{h}_{(c,d)}=\begin{pmatrix}
	\frac{1}{1+td}&0\\ tc & 1+td
\end{pmatrix}\in \mathit{SL}_2(A_n).
\]
Since the natural map $\mathit{SL}_2(A)\to \mathit{SL}_2(A_n)$ is surjective, 
we can take a lift $h_{(c,d)}\in \Gamma_1(t)$ of $\bar{h}_{(c,d)}$ by this map. Then the set
\[
\{h_{(c,d)}\mid c,d\in A_{n-1}\}
\] 
forms a complete set of representatives of $\Gamma_1(t^n)\backslash \Gamma_1(t)$.

Note that for
\[
SB(A)=\left\{\begin{pmatrix}
	*&*\\0&*
\end{pmatrix}\in\mathit{SL}_2(A)\right\},
\]
the map
\[
\Gamma_1(t)\backslash \mathit{SL}_2(A)/SB(A)\to \Gamma_1(t)\backslash \bP^1(K),\quad \gamma\mapsto \gamma(\infty)
\]
is bijective. Hence we obtain
\[
\Cusps(\Gamma_1(t))=\{\infty,0\}.
\]

Consider the natural map 
\[
\Cusps(\Gamma_1(t^n))\to \Cusps(\Gamma_1(t)).
\]
For $\bullet\in\{\infty,0\}$, we denote by $\Cusps_\bullet(\Gamma_1(t^n))$ the inverse image of $\bullet$ by this map.
Then we have a bijection
\[
\Gamma_1(t^n)\backslash \Gamma_1(t)/\Stab(\Gamma_1(t),\bullet)\to \Cusps_\bullet(\Gamma_1(t^n)),\quad \gamma\mapsto \gamma(\bullet).
\]
From the equalities
\begin{align*}
\Stab(\Gamma_1(t),\infty)&=\left\{\begin{pmatrix}
	1 & b\\0&1
\end{pmatrix} \ \middle|\ b\in A\right\},\\
\Stab(\Gamma_1(t),0)&=\left\{\begin{pmatrix}
	1 & 0\\tc&1
\end{pmatrix} \ \middle|\ c\in A\right\},
\end{align*}
we can show the following lemma.

\begin{lem}\label{LemCusps}
	\begin{enumerate}
		\item\label{LemCusps_infty}
		Let $\Lambda_\infty$ be a subset of $A_{n-1}^2$ which forms a complete set of representatives for the equivalence relation
		\[
		(c,d)\sim (c',d')\Leftrightarrow c=c' \text{ and }d'-d\in cA_{n-1}. 
		\]
		Then the set
		\[
		\{h_{(c,d)}(\infty)\mid (c,d)\in \Lambda_\infty\}
		\]
		forms a complete set of representatives of $\Cusps_\infty(\Gamma_1(t^n))$.
		\item\label{LemCusps_zero} The set
		\[
		\{h_{(0,d)}(0)\mid d\in A_{n-1}\}
		\]
		forms a complete set of representatives of $\Cusps_0(\Gamma_1(t^n))$.
	\end{enumerate}
\end{lem}

\begin{lem}\label{LemUnif}
	Let $(c,d)$ be any element of $A_{n-1}^2$. Put $m=\bar{v}_t(c)\in [0,n-1]$.
	\begin{enumerate}
		\item\label{LemUnif_inf} For $s=h_{(c,d)}(\infty)$, we have
		\[
		\mathfrak{b}_s=(t^{n-1-m}),\quad u_s(z)=u_{n-1-m}(z)=\frac{1}{t^{n-1-m}}u\left(\frac{z}{t^{n-1-m}}\right).
		\]
		
		\item\label{LemUnif_zero} For $s=h_{(0,d)}(0)$, we have 
		\[
		\mathfrak{b}_s=(t^n),\quad u_s(z)=u_n(z)=\frac{1}{t^n}u\left(\frac{z}{t^n}\right).
		\]
	\end{enumerate}
	
\end{lem}
\begin{proof}
	For any $x\in A$, the element
	\begin{equation}\label{EqnGammaInf}
		h_{(c,d)}
	\begin{pmatrix}
		1&x\\0&1
	\end{pmatrix}
h_{(c,d)}^{-1}\in \mathit{SL}_2(A)
\end{equation}
is congruent modulo $t^n$ to
\[
\begin{pmatrix}
	1-\frac{tcx}{1+td} & \frac{x}{(1+td)^2}\\
	-t^2c^2x & 1+\frac{tcx}{1+td}
\end{pmatrix}
\]
and thus the element of (\ref{EqnGammaInf}) lies in $\Gamma_1(t^n)$ if and only if
\[
\bar{v}_t(x)\geq \max\{n-1-m,n-2-2m\}=n-1-m,
\]
which yields (\ref{LemUnif_inf}).

For (\ref{LemUnif_zero}), observe
\[
h_{(0,d)}(0)=h_{(0,d)}J(\infty),\quad J=\begin{pmatrix}
	0 & -1\\1& 0
\end{pmatrix}.
\]
Since 
\[
h_{(0,d)}J
\begin{pmatrix}
	1&x\\0&1
\end{pmatrix}
J^{-1}h_{(0,d)}^{-1}\equiv \begin{pmatrix}
	1&0\\-x(1+td)^2&1
\end{pmatrix}\bmod t^n,
\]
the element of the left-hand side lies in $\Gamma_1(t^n)$ if and only if $x\in (t^n)$.
This concludes the proof.
\end{proof}

%---------------------------------------------------------------------
%---------------------------------------------------------------------

%---------------------------------------------------------------------
%---------------------------------------------------------------------

\subsection{Hecke operators of level $\Gamma_1(t^n)$}\label{SubsecHecketn}

%---------------------------------------------------------------------
%---------------------------------------------------------------------

\begin{lem}\label{LemHeckeOp}
	For any $f\in S_k(\Gamma_1(t^n))$, monic irreducible polynomial $\frem\in A$ and $d\in A_{n-1}$, we have
	\[
	(T_\frem f)|_k{h_{(0,d)}}=\left\{\begin{array}{ll}
		\sum_{\deg(\beta)<\deg(\frem)} f|_k{h_{(0,d)}}{\xi_\beta} & (\frem=t),\\
		\sum_{\deg(\beta)<\deg(\frem)} f|_k{h_{(0,d)}}{\xi_{\frem,\beta}}+f|_k{h_{(0,d)}}{\xi_{\frem,\diamond}} & (\frem \neq t).
	\end{array}\right.
	\]
	Moreover, when $\frem\neq t$, we can write
	\[
	(f|_k{h_{(0,d)}}{\xi_{\frem,\diamond}})(z)=\sum_{i\geq 2}c_i u(z)^i,\quad c_i\in \bC_\infty
	\]
	if $|u(z)|$ is sufficiently small.
\end{lem}
\begin{proof}
	Since $f|_k{h_{(0,d)}}=\langle 1+td\rangle_{t^n}f$, Lemma \ref{LemDiaHecke} shows the former assertion.
	
	Let us show the latter assertion for $\frem\neq t$. We have
	\[
	(f|_k{h_{(0,d)}}{\xi_{\frem,\diamond}})(z)=\frem^{k-1}(f|_k{h_{(0,d)}\eta_{\frem,\diamond}})(\frem z).
	\]
	For any $x\in A$, observe
	\[
	h_{(0,d)}\eta_{\frem,\diamond}\begin{pmatrix}
		1 & x\\0& 1
	\end{pmatrix}
	(h_{(0,d)}\eta_{\frem,\diamond})^{-1}\in \Gamma_1(t^n),
	\]
	which shows that the uniformizer at the cusp $h_{(0,d)}\eta_{\frem,\diamond}(\infty)$ is $u(z)$. Then we can write
	\[
	(f|_k{h_{(0,d)}\eta_{\frem,\diamond}})(z)=\sum_{i\geq 1} b_iu(z)^i,\quad b_i\in \bC_\infty,
	\]
	and the assertion follows from Lemma \ref{LemCoeffHecke} (\ref{LemCoeffHecke-dia}).
\end{proof}

\begin{lem}\label{LemUtCusp}
	Let $\beta\in\bF_q$ and $(c,d)\in A_{n-1}^2$ be any elements.
	\begin{enumerate}
		\item\label{LemUtCusp-inf} $\xi_\beta h_{(c,d)}\in \Gamma_1(t^n)h_{(tc,d-\beta c)}\xi_\beta$.
		\item\label{LemUtCusp-inf0}
		If $\beta\neq 0$, then 
		\[
		\xi_\beta h_{(c,d)}J\in \Gamma_1(t^n)h_{(\beta^{-1}(1+td),d-\beta c)}\begin{pmatrix}
			1 &0 \\ 0& t
		\end{pmatrix}
	\begin{pmatrix}
		\beta& -1\\0&\beta^{-1}
	\end{pmatrix}.
		\]
		\item\label{LemUtCusp-0}
		$\xi_0 h_{(c,d)}J\in \Gamma_1(t^n)h_{(tc,d)}J\begin{pmatrix}
			t &0 \\ 0& 1
		\end{pmatrix}$.
	\end{enumerate}
\end{lem}
\begin{proof}
	Write
	\[
	h_{(c,d)}=\begin{pmatrix}
		P & t^n Q\\ tR & S
	\end{pmatrix},\quad P,Q,R,S\in A. 
	\]
	Since $S\equiv P\bmod t$, we have $t^{-1}(S-P)\in A$ and the element
	\[
	\xi_\beta h_{(c,d)}\xi_\beta^{-1}=\begin{pmatrix}
		P+t\beta R & \beta\left(\frac{S-P}{t}\right)-\beta^2 R+t^{n-1}Q\\
		t^2 R & S-t\beta R
		\end{pmatrix}\in \Gamma_1(t)
	\]
	satisfies
	\[
	\xi_\beta h_{(c,d)}\xi_\beta^{-1} \equiv \begin{pmatrix}
		* & *\\
		t^2 c & 1+t(d-\beta c)
	\end{pmatrix}\bmod t^n,
	\]
	which shows (\ref{LemUtCusp-inf}).
	
	For (\ref{LemUtCusp-inf0}), the matrix $\xi_\beta h_{(c,d)}J$ equals
	\[
	\begin{pmatrix}
		t^n\beta^{-1} Q+S& \beta\left(\frac{S-P}{t}\right)+t^{n-1}Q-\beta^2 R\\t\beta^{-1}S&S-t\beta R
	\end{pmatrix}\begin{pmatrix}
		1& 0\\0&t
	\end{pmatrix}\begin{pmatrix}
		\beta& -1\\0&\beta^{-1}
	\end{pmatrix}.
\]
\begin{comment}
	\begin{align*}
	\begin{pmatrix}
		t^n Q+\beta S& -P-t\beta R\\tS&-t^2R
	\end{pmatrix}&=\begin{pmatrix}
		t^n\beta^{-1} Q+S& \beta(S-P)+t^nQ-t\beta^2 R\\t\beta^{-1}S&tS-t^2\beta R
	\end{pmatrix}\begin{pmatrix}
		\beta& -1\\0&\beta^{-1}
	\end{pmatrix}\\
	&=\begin{pmatrix}
		t^n\beta^{-1} Q+S& \beta\left(\frac{S-P}{t}\right)+t^{n-1}Q-\beta^2 R\\t\beta^{-1}S&S-t\beta R
	\end{pmatrix}\begin{pmatrix}
		1& 0\\0&t
	\end{pmatrix}\begin{pmatrix}
		\beta& -1\\0&\beta^{-1}
	\end{pmatrix}.
	\end{align*}
	\end{comment}
	The first matrix lies in $\Gamma_1(t)$, and it is congruent modulo $t^n$ to
	\[
	\begin{pmatrix}
		*& * \\t\beta^{-1}(1+td)& 1+t(d-\beta c)
	\end{pmatrix}.
	\]
	Hence this matrix is contained in $\Gamma_1(t^n)h_{(\beta^{-1}(1+td),d-\beta c)}$ and (\ref{LemUtCusp-inf0}) follows.
	
	For (\ref{LemUtCusp-0}), the matrix $\xi_0 h_{(c,d)}J$ equals
	\begin{align*}
		\begin{pmatrix}
			1& 0\\0&t
		\end{pmatrix}
		\begin{pmatrix}
			P& t^n Q\\tR&S
		\end{pmatrix}
		\begin{pmatrix}
			0& -1\\1&0
		\end{pmatrix}=\begin{pmatrix}
			P& t^{n-1} Q\\t^2 R&S
		\end{pmatrix}J\begin{pmatrix}
			t& 0\\0&1
		\end{pmatrix}.
	\end{align*}
	The first matrix of the right-hand side lies in $\Gamma_1(t)$, and it is congruent modulo $t^n$ to
	\[
	\begin{pmatrix}
		*& * \\t^2c& 1+td
	\end{pmatrix},
	\]
	from which (\ref{LemUtCusp-0}) follows.
\end{proof}

\begin{lem}\label{LemDiamond}
	Let $a,c,d\in A_{n-1}$ be any elements. Take any lift $\fra\in A$ of $1+ta\in A_n$.
	Then we have
	\[
	\eta_{\fra,\diamond}h_{(c,d)}\in \Gamma_1(t^n)h_{((1+ta)c,a+d+tad)}.
	\]
\end{lem}
\begin{proof}
	Since $\fra\equiv 1\bmod t$, the matrix $\eta_{\fra,\diamond}$ lies in $\Gamma_1(t)$. 
	Thus the lemma follows from
	\begin{align*}
		\eta_{\fra,\diamond}h_{(c,d)}&\equiv \begin{pmatrix}
	* & *\\
	t(1+ta)c & (1+ta)(1+td)
\end{pmatrix}\bmod t^n.
	\end{align*}
\end{proof}

%---------------------------------------------------------------------
%---------------------------------------------------------------------

\subsection{Unramified double cuspforms}\label{SubsecDDC}

Put
\[
S'_k=\{f\in S_k\mid \ord\left(h_{(0,d)}(\infty),f\right)\geq 2\text{ for any }d\in A_{n-1}\}.
\]

\begin{lem}\label{LemPDCStable}
	$S'_k$ is stable under all Hecke operators.
\end{lem}
\begin{proof}
	Let $f$ be any element of $S'_k$ and $\frem\in A$ any monic irreducible polynomial. 
	By Lemma \ref{LemUnif} (\ref{LemUnif_inf}) the uniformizer at the cusp $h_{(0,d)}(\infty)$ is $u(z)$ and we can write
\[
(f|_k{h_{(0,d)}})(z)=\sum_{i\geq 2} a_i u(z)^i,\quad a_i\in \bC_\infty.
\]
Then Lemma \ref{LemCoeffHecke} (\ref{LemCoeffHecke-beta2}) shows that the term 
	\[
	\sum_{\deg(\beta)<\deg(\frem)} f|_k{h_{(0,d)}}{\xi_{\frem,\beta}}
	\]
	in the equality of Lemma \ref{LemHeckeOp} has 
	no 
	linear term 
	of $u(z)$.
	Thus the lemma follows from the latter assertion of Lemma \ref{LemHeckeOp}.
\end{proof}

For any $f\in S_k$ and $d\in A_{n-1}$, we write
\[
(f|_kh_{(0,d)})(z)=\sum_{i\geq 1} a_iu(z)^i,\quad a_i\in \bC_\infty
\]
and put $L_d(f)=a_1$.
Then the $\bC_\infty$-linear map
\[
L:S_k/S'_k\to \bigoplus_{d\in A_{n-1}}\bC_\infty, \quad f\mapsto (L_d(f))_d
\]
is injective.

\begin{lem}\label{LemQuotDim}
	\[
	\dim_{\bC_\infty} S_k/S'_k =q^{n-1}.
	\]
	In particular, the map $L$ is bijective.
\end{lem}
\begin{proof}
	We denote $\Cusps(\Gamma_1(t^n))$ also by $\Cusps$. 
	By Lemma \ref{LemCusps} (\ref{LemCusps_infty}), the points
	\[
	h_{(0,d)}(\infty),\quad d\in A_{n-1}
	\]
	form a subset $\Cusps'$ of cardinality $q^{n-1}$ of $\Cusps$. 
	We abusively identify $\Cusps$ and $\Cusps'$ with the reduced divisors they define on the 
	Drinfeld modular curve $X_1(t^n)_{\bC_\infty}$ over $\bC_\infty$, and
	put $D=\Cusps+\Cusps'$.
	Let $g$ be the genus of $X_1(t^n)_{\bC_\infty}$ and $h$ the 
	number of cusps. Since $0\in \Cusps\setminus \Cusps'$, we have $h>q^{n-1}$.
	
	Let $\bar{\omega}$ be the Hodge bundle on $X_1(t^n)_{\bC_\infty}$,
	so that $\deg(\bar{\omega}^{\otimes 2})=2g-2+2h$ and $\deg(\bar{\omega})\geq 0$ (see for example \cite[Corolary 4.2]{Ha_HTT} with $\Delta=\{1\}$).
	For $k\geq 2$, we have
	\begin{align*}
	\deg(\bar{\omega}^{\otimes k}(-D))&=k\deg(\bar{\omega})-\deg(D)\\
	&=(k-2)\deg(\bar{\omega})+2g-2+h-q^{n-1}\geq 2g-1.
	\end{align*}
	Since $S'_k$ can be identified with $H^0(X_1(t^n)_{\bC_\infty},\bar{\omega}^{\otimes k}(-D))$, the Riemann-Roch theorem implies
	\[
	\dim_{\bC_\infty} S'_k=\deg(\bar{\omega}^{\otimes k}(-D))+1-g=(k-1)(g-1+h)-q^n.
	\]
	From $\dim_{\bC_\infty}S_k=(k-1)(g-1+h)$ \cite[Proposition 5.4]{Boeckle}, we obtain $\dim_{\bC_\infty}S_k/S'_k=q^{n-1}$. Since the both sides of the 
	injection $L$ 
	have the same 
	dimension, it 
	is a bijection.
\end{proof}

\begin{lem}\label{LemQuotTriv}
	All Hecke operators act trivially on $S_k/S'_k$.
\end{lem}
\begin{proof}
	Let $\frem\in A$ be any monic irreducible polynomial. 
	Take any $f\in S_k$.
	By Lemma \ref{LemCoeffHecke} and Lemma \ref{LemHeckeOp},
	we obtain $L_d(T_\frem f)=L_d(f)$ for any $d\in A_{n-1}$ and the injectivity of the map 
	$L$ shows $T_\frem f \equiv f \bmod S'_k$. This concludes the proof.
\end{proof}

%---------------------------------------------------------------------
%---------------------------------------------------------------------

%---------------------------------------------------------------------
%---------------------------------------------------------------------

\subsection{Nilpotency of $U_t$ on $S'_k/S_k^{(2)}$}\label{SubsecNilp}

%---------------------------------------------------------------------
%---------------------------------------------------------------------

For any integer $i$, put
\[
C_i=\{(c,d)\in A_{n-1}^2\mid \bar{v}_t(c)\geq i\}.
\]
To study the $U_t$-action on $S'_k$, we define
\[
	S'_{k,i}=\{f\in S_k\mid \ord(h_{(c,d)}(\infty),f)\geq 2\text{ for any }(c,d)\in C_i \}
\]
so that
\[
S'_k=S'_{k,n-1}\supseteq S'_{k,n-2}\supseteq\cdots \supseteq S'_{k,0}=S'_{k,-1}\supseteq S_k^{(2)}.
\]

\begin{prop}\label{PropMidNilp}
	Let $i\in [0,n-1]$ be any integer. 
	\begin{enumerate}
		\item\label{PropMidNilp-inf} $U_t(S'_{k,i})\subseteq S'_{k,i-1}$.
		\item\label{PropMidNilp-inf0} $U_t(S'_{k,0})\subseteq S^{(2)}_k$.
	\end{enumerate}
In particular, the operator $U_t$ acting on $S'_k/S^{(2)}_k$ is nilpotent.
\end{prop}
\begin{proof}
	For the assertion (\ref{PropMidNilp-inf}), take any $f\in S'_{k,i}$ and $(c,d)\in C_{i-1}$. 
	We need to show
	\begin{equation}\label{EqnPropMidNilp}
		\ord(h_{(c,d)}(\infty),U_tf)\geq 2.
	\end{equation}
	Since the case of $c=0$ follows from Lemma \ref{LemPDCStable}, we may assume $c\neq 0$. 
	Put $m=\bar{v}_t(c)$. 
	For any $\beta\in \bF_q$, we have $(tc,d-\beta c)\in C_i$ and the assumption yields $\bar{v}_t(tc)=m+1$. 
	By Lemma \ref{LemUnif} (\ref{LemUnif_inf}), we can write
	\[
	(f|_k h_{(tc, d-\beta c)})(z)=\sum_{j\geq 2} a_j^{(\beta)}u_{n-2-m}(z)^j,\quad a_j^{(\beta)}\in \bC_\infty
	\]
	and Lemma \ref{LemUtCusp} (\ref{LemUtCusp-inf}) yields
	\begin{align*}
		((U_t f)|_k h_{(c,d)})(z)&=\sum_{\beta\in \bF_q} (f|_k\xi_\beta h_{(c,d)})(z)=\sum_{\beta\in \bF_q} (f|_k h_{(tc,d-\beta c)}\xi_\beta)(z)\\
		&=\frac{1}{t}\sum_{\beta\in \bF_q}\sum_{j\geq 2} a_j^{(\beta)}u_{n-2-m}\left(\frac{z+\beta}{t}\right)^j.
	\end{align*}
	Since the uniformizer at $h_{(c,d)}(\infty)$ is $u_{n-1-m}(z)$, Lemma \ref{LemUnifHecke} gives the inequality (\ref{EqnPropMidNilp}).

	Let us show the assertion (\ref{PropMidNilp-inf0}). Take any $f\in S'_{k,0}$ and $d\in A_{n-1}$. 
	Since we already know $U_tf\in S'_{k,0}$ by (\ref{PropMidNilp-inf}), it is enough to show
	\begin{equation}\label{EqnPropMidNilp0}
	\ord(h_{(0,d)}(0),U_tf)\geq 2.
	\end{equation}
By Lemma \ref{LemUnif} (\ref{LemUnif_zero}), the uniformizer at $h_{(0,d)}(0)=h_{(0,d)}J(\infty)$ is $u_n(z)$. 
	
	Consider the equality
	\begin{equation}\label{EqnMidNilp-inf0}
		(U_tf)|_k h_{(0,d)}J=\sum_{\beta\in \bF_q^\times} f|_k\xi_\beta h_{(0,d)}J+ f|_k \xi_0h_{(0,d)}J.
	\end{equation} 
	For the first term in the right-hand side of (\ref{EqnMidNilp-inf0}), we have 
	\[
	\bar{v}_t(\beta^{-1}(1+td))=0
	\] 
	and 
	by Lemma \ref{LemUnif} (\ref{LemUnif_inf}) we can write
	\[
	(f|_k h_{(\beta^{-1}(1+td),d)})(z)=\sum_{j\geq 2} a_j u_{n-1}(z)^j,\quad a_j\in \bC_\infty.
	\]
	Then Lemma \ref{LemUtCusp} (\ref{LemUtCusp-inf0}) gives
	\begin{align*}
		(f|_k \xi_\beta h_{(0,d)}J)(z)&=t^{k-1}(\beta^{-1}t)^{-k} \sum_{j\geq 2} a_j u_{n-1}\left(\frac{\beta z-1}{\beta^{-1}t}\right)^j\\
	&=\frac{\beta^k}{t} \sum_{j\geq 2} a_j \beta^{-2j} u_{n-1}\left(\frac{z-\beta^{-1}}{t}\right)^j
		\end{align*}
	and by Lemma \ref{LemUnifHecke} this term lies in $u_n(z)^2\bC_\infty[[u_n(z)]]$.
	
	For the second term in the right-hand side of (\ref{EqnMidNilp-inf0}), write
	\[
	(f|_k h_{(0,d)}J)(z)=\sum_{j\geq 1} a_j u_n(z)^j,\quad a_j\in \bC_\infty.
	\]
	By Lemma \ref{LemUtCusp} (\ref{LemUtCusp-0}), we have
	\[
	(f|_k \xi_0h_{(0,d)}J)(z)=t^{k-1}(f|_k h_{(0,d)}J)(tz)=t^{k-1}\sum_{j\geq 1} a_j u_n(tz)^j.
	\]
	Since Lemma \ref{LemCoeffHecke} (\ref{LemCoeffHecke-dia}) shows
	\[
	u_n(tz) \in u_n(z)^2\bC_\infty[[u_n(z)]],
	\]
	we obtain the inequality (\ref{EqnPropMidNilp0}).
	This concludes the 
	proof of the proposition.
\end{proof}

Recall that we fixed an embedding $\iota_t:\bar{K}\to \bC_t$. We say $\lambda\in \bar{K}$ is a $t$-adic unit if $\iota_t(\lambda)\in \cO^\times_{\bC_t}$.

\begin{thm}\label{ThmNilpV}
For any integer $k\geq 2$, the following are equivalent.
\begin{enumerate}
	\item\label{ThmNilpV-DCk} $U_t$ acting on $S_k^{(2)}(\Gamma_1(t^n))$ has no $t$-adic unit eigenvalue.
	\item\label{ThmNilpV-k} $U_t$ acting on $S'_k$ has no $t$-adic unit eigenvalue.
	\item\label{ThmNilpV-dim} $\dim_{\bC_\infty}S_k^\ord(\Gamma_1(t^n))\leq q^{n-1}$.
	\item\label{ThmNilpV-DC2} $U_t$ acting on $S_2^{(2)}(\Gamma_1(t^n))$ is nilpotent.
	\item\label{ThmNilpV-2} $U_t$ acting on $S'_2$ is nilpotent.
	\item\label{ThmNilpV-dim2} $\dim_{\bC_\infty}S_2^\ord(\Gamma_1(t^n))\leq q^{n-1}$.
\end{enumerate}	
If these equivalent conditions hold, then for any $k\geq 2$ we have
\[
\dim_{\bC_\infty}S_k^\ord(\Gamma_1(t^n))=q^{n-1}
\]
and all Hecke operators act trivially on $S_k^\ord(\Gamma_1(t^n))$.
\end{thm}
\begin{proof}
	
	The equivalence of (\ref{ThmNilpV-DCk})--(\ref{ThmNilpV-dim}) follows from Lemma \ref{LemQuotDim}, Lemma \ref{LemQuotTriv} and Proposition 
	\ref{PropMidNilp}. 
	By \cite[(2.6) and Proposition 2.2]{Ha_DMBF}, any eigenvalue of $U_t$ acting on $S_2(\Gamma_1(t^n))$ is algebraic over $\bF_q$.
	Thus $U_t$ acts on a subspace of $S_2(\Gamma_1(t^n))$ without $t$-adic unit eigenvalue if and only if the action is nilpotent.
	This shows the equivalence of (\ref{ThmNilpV-DC2})--(\ref{ThmNilpV-dim2}). The equivalence of (\ref{ThmNilpV-dim}) and (\ref{ThmNilpV-dim2})
	follows from \cite[Proposition 3.4 (1)]{Ha_DMBF}. 
	
	%By Proposition \ref{Prop2Const}, these conditions are also equivalent to 
	%(\ref{ThmNilpV-DCF})--(\ref{ThmNilpV-2F}).
	
	If these conditions hold, then we have $\dim_{\bC_\infty}S_k^\ord(\Gamma_1(t^n))=q^{n-1}$
	and 
	the natural map
	\[
	S_k^\ord(\Gamma_1(t^n))\to S_k/S'_k
	\]
	is an isomorphism compatible with all Hecke operators. Now the last assertion follows from Lemma \ref{LemQuotTriv}. 
\end{proof}

Since $X_1(t)_{\bC_\infty}$ is of genus zero, we have $S_2^{(2)}(\Gamma_1(t))=0$ and the nilpotency of $U_t$ acting on it holds trivially.
Thus Theorem \ref{ThmNilpV} yields the following corollary, which reproves \cite[Lemma 2.4]{Ha_GVt} and \cite[Proposition 4.3]{Ha_DMBF}
without using the theory of $A$-expansions \cite{Pet} or Bandini-Valentino's formula \cite[(4.2)]{BV0}.

\begin{cor}\label{Cor1}
	For any $k\geq 2$, we have
	\[
	\dim_{\bC_\infty}S_k^\ord(\Gamma_1(t))=1
	\]
	and all Hecke operators act trivially on $S_k^\ord(\Gamma_1(t))$.
\end{cor}

Note that by \cite[Corollary 5.7]{GN} the genus of $X_1(t^n)_{\bC_\infty}$ is 
\[
1+q^{2n-2}-(n+1)q^{n-1}+(n-1)q^{n-2}
\]
and for $n\geq 2$ it is zero only if $n=q=2$.

%---------------------------------------------------------------------

%---------------------------------------------------------------------

%---------------------------------------------------------------------

\section{Freeness and triviality}\label{SecTriv}

In this section, we prove the triviality of the Hecke action on $S_k(\Gamma_1(t^n))$ for any $k\geq 2$ and $n\geq 1$
(Theorem \ref{ThmTriv}).
Put $\Theta_n=1+tA_n\subseteq A_n^\times$. The key point of the proof is to show that 
$S_2(\Gamma_1(t^n))$, which we consider as a $\bC_\infty[\Theta_n]$-module via the diamond operator, is
the direct sum of copies of $\bC_\infty[\Theta_n]$ (Proposition \ref{PropFree}). 
For this, we need a description of $S_2(\Gamma_1(t^n))$ using harmonic 
cocycles on the Bruhat-Tits 
tree.

%---------------------------------------------------------------------

\subsection{Bruhat-Tits tree and $\Gamma_1(t^n)$}\label{SubsecBT}

We consider $K_\infty^2$ as the set of row vectors, and define an action $\circ$ of $\mathit{GL}_2(K_\infty)$ on $K_\infty^2$ by \[
\gamma\circ (x_1,x_2)=(x_1,x_2)\gamma^{-1}.
\]
Let $\cT$ be the Bruhat-Tits tree for $\mathit{SL}_2(K_\infty)$. 
Recall that the set $\cT_0$ of vertices of $\cT$ is by definition the set of $K_\infty^\times$-equivalence classes of $\cO_{K_\infty}$-lattices in $K_\infty^2$,
where $\cO_{K_\infty}$ is the ring of integers of $K_\infty$. 
The action $
\circ$ 
induces an action of $\mathit{GL}_2(K_\infty)$ on the tree $\cT$, and also on the oriented tree $\cT^o$ associated to $\cT$.
We denote by $\cT^o_1$ the set of oriented edges. For any $e\in \cT^o_1$, the origin, the terminus and the opposite edge of $e$ are denoted by $o(e)$, 
$t(e)$ and $-e$, respectively. Then the group $\{\pm 1\}$ acts on $\cT^o_1$ by $(-1)e=-e$, which commutes with the action of $\mathit{GL}_2(K_\infty)$.

Put $\pi=1/t$, which is a uniformizer of $K_\infty$. For any integer $i$, let $v_i$ be the class of the lattice $\cO_{K_\infty}(\pi^{i},0)\oplus  \cO_{K_\infty}
(0,1)$. Then we have $\begin{pmatrix}
\pi^{-i}&0\\0&1
\end{pmatrix}v_0=v_i$. We denote by $e_i$ the oriented edge with origin $v_i$ and terminus $v_{i+1}$.

For any subgroup $\Gamma$ of $\mathit{SL}_2(A)$, we say $e\in \cT^o_1$ is $\Gamma$-stable if $\Stab(\Gamma,e)=\{1\}$, and $\Gamma$-unstable otherwise. 
We define $\Gamma$-stability of a vertex similarly.
The set of $\Gamma$-stable edges is denoted by $\cT_1^{o,\Gamma\text{-}\mathrm{st}}$.
For $\Gamma=\Gamma_1(t)$, we know \cite[\S7]{LM} that the set of $\Gamma_1(t)$-stable edges is equal to $\Gamma_1(t)J(\pm e_0)$ with
\[
J=\begin{pmatrix}
	0&-1\\1&0
\end{pmatrix}.
\]
By \cite[Ch.~II, \S1.2, Corollary]{Se}, this shows:

\begin{lem}\label{LemSt1}
	A complete set of representatives of $\Gamma_1(t^n)\backslash \cT_1^{o,\Gamma_1(t)\text{-}\mathrm{st}}/\{\pm 1\}$ is given by
	\[
	\Lambda_{1,n}=\{{h}_{(c,d)}J e_0\mid c,d\in A_{n-1}\}.
	\]
\end{lem}

%---------------------------------------------------------------------
%---------------------------------------------------------------------

\subsection{Harmonic cocycles}\label{SubsecHC}

%---------------------------------------------------------------------
%---------------------------------------------------------------------

In this subsection, we recall a description of Drinfeld cuspforms using harmonic cocycles due to Teitelbaum \cite{Tei}, following \cite{Boeckle} and 
\cite{Ha_DMBF}.

Let $k\geq 2$ be any integer. We denote by $H_{k-2}(\bC_\infty)$ the $\bC_\infty$-subspace of homogeneous polynomials of degree $k-2$ in 
the polynomial ring $\bC_\infty[X,Y]$. We consider the left action $\circ$ of $\mathit{GL}_2(K)$
on it defined by
\[
\gamma\circ(X,Y)=(X,Y)\gamma.
\] 
We put $V_k(\bC_\infty)=\Hom_{\bC_\infty}(H_{k-2}(\bC_\infty),\bC_\infty)$, on which 
$\mathit{GL}_2(K)$ acts naturally.
For $\xi\in \mathit{GL}_2(K)$, $\omega\in V_k(\bC_\infty)$ and $P(X,Y)\in H_{k-2}(\bC_\infty)$, the action is given by
\[
(\xi\circ \omega)(P(X,Y))=\omega(\xi^{-1}\circ P(X,Y))=\omega(P((X,Y)\xi^{-1})).
\]

\begin{dfn}\label{DefHar}
	A map $c:\cT^o_1\to V_k(\bC_\infty)$ is called a harmonic cocycle of weight $k$ over $\bC_\infty$ if the following two conditions hold:
	\begin{enumerate}
		\item For any $v\in \cT_0$, we have
		\[
		\sum_{e\in \cT^o_1,\ t(e)=v}c(e)=0.
		\]
		\item For any $e\in \cT_1^o$, we have $c(-e)=-c(e)$.
		\end{enumerate}
	For any arithmetic subgroup $\Gamma$ of $\mathit{SL}_2(A)$, we say $c$ is $\Gamma$-equivariant if $c(\gamma e)=\gamma\circ c(e)$ for any $\gamma\in 
	\Gamma$ and $e\in \cT^o_1$. We denote the $\bC_\infty$-vector space of $\Gamma$-equivariant harmonic cocycles of weight $k$ over $\bC_\infty$ by 
	$C^\har_k(\Gamma)$.
	\end{dfn}

Let $\Gamma$ be an arithmetic subgroup of $\mathit{SL}_2(A)$ which is $p'$-torsion free. In this case, for any $\Gamma$-unstable vertex $v$,
the group $\Stab(\Gamma,v)$ fixes a unique rational end which we denote by $b(v)$.

\begin{dfn}\label{DefSource}
	A $\Gamma$-stable edge $e'\in \cT^o_1$ is called a $\Gamma$-source of an edge $e\in \cT^o_1$ if the following conditions hold.
	\begin{enumerate}
		\item If $e$ is $\Gamma$-stable, then $e'=e$.
		\item If $e$ is $\Gamma$-unstable, then  a vertex $v$ of $e'$ is $\Gamma$-unstable, $e$ lies on the unique half line from $v$ to $b(v)$ and $e$ has the 
		same orientation as $e'$ with respect to this half line.
		\end{enumerate}
	The set of $\Gamma$-sources of $e$ is denoted by $\src_\Gamma(e)$.
	\end{dfn}
For any harmonic cocycle $c:\cT^o_1\to V_k(\bC_\infty)$ of weight $k$ over $\bC_\infty$, we have
\begin{equation}\label{EqnSrc}
c(e)=\sum_{e'\in\src_\Gamma(e)}c(e').
\end{equation}

We denote by $S_k(\Gamma)$ the $\bC_\infty$-vector space of Drinfeld cuspforms of level $\Gamma$ and weight $k$. 
Then, for any rigid analytic function $f$ on $\Omega$ and $e\in \cT^o_1$, 
Teitelbaum defined an element $\Res(f)(e)\in V_k(\bC_\infty)$, which gives
a natural isomorphism of $\bC_\infty$-vector spaces
\begin{equation}\label{EqnIsomRes}
\Res_\Gamma: S_k(\Gamma)\to C^\har_k(\Gamma),\quad f\mapsto (e\mapsto \Res(f)(e))
\end{equation}
\cite[Theorem 16]{Tei}. Note that we are following the normalization in \cite[Theorem 5.10]{Boeckle}. Moreover, by \cite[(17)]{Boeckle}, the slash operator
can be read off via the corresponding harmonic cocycle by
\begin{equation}\label{EqnBoeckle}
\Res(f|_k\gamma)(e)=\gamma^{-1}\circ \Res(f)(\gamma e).
\end{equation}

Teitelbaum gave another description of Drinfeld cuspforms using the Steinberg module $\St$, which is defined as the kernel of the augmentation map
\[
\St=\Ker(\bZ[\bP^1(K)]\to \bZ).
\]
It admits a natural right $\mathit{GL}_2(K)$-action via
\[
(\gamma,(x:y))\mapsto (x:y)\gamma.
\]
Then, for any arithmetic subgroup $\Gamma$ of $\mathit{SL}_2(A)$ which is $p'$-torsion free, 
\cite[p.~506]{Tei} gives a $\bC_\infty$-linear isomorphism
\begin{equation}\label{EqnHarSt}
C^\har_k(\Gamma)\to \St\otimes_{\bZ[\Gamma]}V_k(\bC_\infty).
\end{equation}

\begin{lem}\label{LemDim}
\[
\dim_{\bC_\infty} C^\har_2(\Gamma_1(t^n))=q^{2(n-1)}.
\]
\end{lem}
\begin{proof}
	The isomorphism (\ref{EqnHarSt}) and \cite[Lemma 3.6]{Ha_DMBF} show that the dimension is $[\Gamma_1(t):\Gamma_1(t^n)]$, which equals $\sharp A^2_{n-1}
	=q^{2(n-1)}$.
	\end{proof}

\begin{lem}\label{LemHar2}
	Let $c$ be any element of $C^\har_2(\Gamma_1(t^n))$.
	\begin{enumerate}
		\item\label{LemHar2-Triv} For any $\gamma\in\Gamma_1(t^n)$ and $e\in \cT^o_1$, we have $c(\gamma e)=c(e)$.
		\item\label{LemHar2-St1} $c$ is determined by its restriction to the subset $\Lambda_{1,n}$ of Lemma \ref{LemSt1}.
		\end{enumerate}
	\end{lem}
\begin{proof}
	Since the group $\mathit{GL}_2(K)$ acts trivially on $V_2(\bC_\infty)=\bC_\infty$, we have $c(\gamma e)=\gamma\circ c(e)=c(e)$ and
	the assertion (\ref{LemHar2-Triv}) follows. 
	
	For the assertion (\ref{LemHar2-St1}), it suffices to show that if the restriction of $c$ to 
	$\Lambda_{1,n}$ is zero, then $c(e)=0$ for any $e\in \cT_1^o$. By (\ref{EqnSrc}), 
	we may assume that $e$ is $\Gamma_1(t)$-stable. 
	Then it is written as $e=\pm \gamma e'$ with some $e'\in \Lambda_{1,n}$ and $\gamma\in \Gamma_1(t^n)$,
	which yields 
	$c(e)=\pm \gamma\circ c(e')=0$. This concludes the proof.
	\end{proof}

\begin{cor}\label{CorExist}
	The $\bC_\infty$-linear map
	\[
	C^\har_2(\Gamma_1(t^n))\to \bigoplus_{e\in \Lambda_{1,n}}\bC_\infty,\quad c\mapsto (c(e))_{e\in\Lambda_{1,n}}
	\]
	is an isomorphism.
	\end{cor}
\begin{proof}
	By Lemma \ref{LemHar2} (\ref{LemHar2-St1}), the map is injective. Since $\sharp \Lambda_{1,n}=q^{2(n-1)}$, Lemma \ref{LemDim} implies that it is an 
	isomorphism.
	\end{proof}

By Corollary \ref{CorExist},
there exists a unique element $[c,d]\in C^\har_2(\Gamma_1(t^n))$ satisfying
\[
[c,d]({h}_{(c',d')}J e_0)=\left\{\begin{array}{ll}1& \text{if }(c',d')=(c,d),\\
	0 & \text{otherwise}.\end{array}\right.
\]
The set $\{[c,d]\mid c,d\in A_{n-1}\}$ forms a basis of the $\bC_\infty$-vector space $C^\har_2(\Gamma_1(t^n))$.

%---------------------------------------------------------------------
%---------------------------------------------------------------------

\subsection{Proof of the main theorem}\label{SubsecFree}

%---------------------------------------------------------------------
%---------------------------------------------------------------------

Consider the subgroup $\Theta_n=1+tA_n$ of $A_n^\times$.
Via the isomorphism $\Res_{\Gamma_1(t^n)}$ of (\ref{EqnIsomRes}), the diamond operator $\langle \alpha\rangle_{t^n}$ acting on $S_2(\Gamma_1(t^n))$ induces an 
operator
on $C^\har_2(\Gamma_1(t^n))$, which we also denote by $\langle \alpha\rangle_{t^n}$.
In particular, the group $\Theta_n$ acts on $C^\har_2(\Gamma_1(t^n))$ via $\alpha\mapsto \langle \alpha\rangle_{t^n}$.

\begin{lem}\label{LemDiaHar}
	For any $a,c,d\in A_{n-1}$, the action of $1+ta\in \Theta_n$ on $[c,d]$ is given by
	\[
	\langle 1+ta \rangle_{t^n}[c,d]=[(1+ta)^{-1}c,(1+ta)^{-1}(d-a)].
	\]
\end{lem}
\begin{proof}
	By (\ref{EqnBoeckle}) and Lemma \ref{LemDiamond}, for any $c',d'\in A_{n-1}$ we have
	\[
		\left(\langle 1+ta \rangle_{t^n}[c,d]\right)(h_{(c',d')}Je_0)=[c,d](h_{((1+ta)c',a+d'+tad')}Je_0),
	\]
	which is equal to one if $(c',d')=((1+ta)^{-1}c,(1+ta)^{-1}(d-a))$ and zero otherwise.
	This concludes the proof.
\end{proof}

\begin{prop}\label{PropFree}
	The $\bC_\infty[\Theta_n]$-module $S_2(\Gamma_1(t^n))$ is isomorphic to the direct sum of $q^{n-1}$ copies of
	$\bC_\infty[\Theta_n]$. 
\end{prop}
\begin{proof}
	It suffices to show the assertion for $C^\har_2(\Gamma_1(t^n))$.
	Take any $(c,d)\in A_{n-1}^2$. We claim that the $\Theta_n$-orbit
	\[
	\{\langle 1+ta\rangle_{t^n}[c,d]\mid a\in A_{n-1}\}
	\]
	of $[c,d]$ is of cardinality $q^{n-1}$. Indeed, if $\langle 1+ta\rangle_{t^n}[c,d]=\langle 1+ta'\rangle_{t^n}[c,d]$
	for some $a,a'\in A_{n-1}$, then Lemma \ref{LemDiaHar} yields
	\[
	(1+ta)^{-1}(d-a)=(1+ta')^{-1}(d-a'),
	\]
	which is equivalent to $(1+td)(a'-a)=0$ and we obtain $a'=a$.
	
	We denote by $V(c,d)$ the $\bC_\infty$-subspace of $C^\har_2(\Gamma_1(t^n))$ spanned by the $\Theta_n$-orbit
	of $[c,d]$.
	Then $V(c,d)$ is stable under the $\Theta_n$-action and
	$\dim_{\bC_\infty}V(c,d)=q^{n-1}$. Consider the map
	\[
	\bC_\infty[\Theta_n]\to V(c,d),\quad \alpha\mapsto \langle\alpha\rangle_{t^n}[c,d].
	\]
	It is a homomorphism of $\bC_\infty[\Theta_n]$-modules which is surjective. Since the both sides have the same 
	dimension, it is an isomorphism. Since the $\bC_\infty$-vector space $C^\har_2(\Gamma_1(t^n))$ is the direct sum of $V(c,d)$'s, the
	proposition follows from Lemma \ref{LemDim}.
\end{proof}

\begin{thm}\label{ThmTriv}
	We have
	\[
	\dim_{\bC_\infty} S_2^\ord(\Gamma_1(t^n))=q^{n-1}
	\]
	and all Hecke operators act trivially on $S_k^\ord(\Gamma_1(t^n))$ for any $k\geq 2$.
\end{thm}
\begin{proof}
	By Theorem \ref{ThmNilpV}, it is enough to show $\dim_{\bC_\infty} S_2^\ord(\Gamma_1(t^n))\leq q^{n-1}$. 
	Put
	\[
	\Gamma_0^p(t^n)=\left\{\gamma\in \mathit{SL}_2(A)\ \middle|\ \gamma\bmod t^n\in\begin{pmatrix}
		1+tA_n & A_n\\ 0 & 1+tA_n
	\end{pmatrix}\right\},
	\]
	as in \cite[\S3]{Ha_DMBF}.
	Then the $\Theta_n$-fixed part of $S_2(\Gamma_1(t^n))$ is $S_2(\Gamma^p_0(t^n))$. 
	Since the Hecke operator $U_t$ commutes with the action of $\Theta_n$ and it is defined by the same formula
	for the levels $\Gamma_1(t^n)$ and $\Gamma_0^p(t^n)$ \cite[\S3.1]{Ha_DMBF}, we see that $S_2^\ord(\Gamma_1(t^n))$ is 
	stable under the $\Theta_n$-action and
	\[
	S_2^\ord(\Gamma_1(t^n))^{\Theta_n}=S_2^\ord(\Gamma^p_0(t^n)),
	\]
	where the right-hand side is the ordinary subspace of $S_2(\Gamma^p_0(t^n))$ defined similarly to
	the case of $S_2(\Gamma_1(t^n))$. Then \cite[Proposition 3.5]{Ha_DMBF} and Corollary \ref{Cor1} yield
	\[
	\dim_{\bC_\infty}S_2^\ord(\Gamma_1(t^n))^{\Theta_n}=\dim_{\bC_\infty}S_2^\ord(\Gamma_0^p(t^n))=\dim_{\bC_\infty}S_2^\ord(\Gamma_1(t))=1.
	\] 
	
	On the other hand, Proposition \ref{PropFree} gives an injection of $\bC_\infty[\Theta_n]$-modules
	\[
	S_2^\ord:=S_2^\ord(\Gamma_1(t^n))\to \bigoplus_{i=1}^{q^{n-1}} V_i,\quad V_i=\bC_\infty[\Theta_n].
	\]
	Let $I$ be the set of integers $M\in [1,q^{n-1}]$ such that there exists an injection of $\bC_\infty[\Theta_n]$-modules
	$S_2^\ord\to \bigoplus_{i=1}^{M} V_i$. Then $I$ is nonempty and let $m$ be its minimal element. 
	
	Now we reduce ourselves to showing $m=1$. Suppose $m>1$ and consider an injection $S_2^\ord\to \bigoplus_{i=1}^{m} V_i$.
	Since $\Theta_n$ is an abelian $p$-group and $\bC_\infty$ contains no non-trivial $p$-power root of unity, 
	Schur's lemma implies that the only irreducible representation of $\Theta_n$ over 
	$\bC_\infty$ is the trivial representation. Since both of
	\[
	S_2^\ord\cap V_1,\quad S_2^\ord\cap \bigoplus_{i=2}^{m} V_i
	\]
	are $\bC_\infty[\Theta_n]$-submodules of $S_2^\ord$, if one of them is non-zero then it contains the trivial representation.
	Since 
	the $\bC_\infty$-vector space $(S_2^\ord)^{\Theta_n}$ is one-dimensional, we see that either of them is zero. Thus either of the induced maps
	\[
	S_2^\ord\to \left(\bigoplus_{i=1}^{m} V_i\right)/V_1\simeq \bigoplus_{i=1}^{m-1} V_i,\quad S_2^\ord\to \left(\bigoplus_{i=1}^{m} V_i\right)/
	\left(\bigoplus_{i=2}^{m} V_i\right)\simeq V_1
	\]
	is injective, which contradicts the minimality of $m$.
	This concludes the proof 
	of the theorem.
\end{proof}

Theorem \ref{ThmNilpV} and Theorem \ref{ThmTriv} yield the following corollary.
\begin{cor}\label{CorUtNilp}
	The operator $U_t$ acting on $S_2^{(2)}(\Gamma_1(t^n))$ is nilpotent.
\end{cor}

\begin{rmk}\label{RmkNilp}
	By Theorem \ref{ThmNilpV}, if we could prove the nilpotency of $U_t$ acting on $S_2^{(2)}(\Gamma_1(t^n))$ directly,
	then Theorem \ref{ThmTriv} would follow.
	As the proof of Theorem \ref{ThmTriv} indicates, the reason we can bypass it is that we know the dimension of $S_2^\ord(\Gamma_1(t))$
	because $X_1(t)_{\bC_\infty}$ is of genus zero.  
	The author has no idea of how to show the nilpotency directly.
\end{rmk}

%---------------------------------------------------------------------
%---------------------------------------------------------------------

%---------------------------------------------------------------------

%---------------------------------------------------------------------

\end{document}